\documentclass{proc-l}

\newtheorem{theorem}{Theorem}[section]
\newtheorem{lemma}[theorem]{Lemma}

\theoremstyle{definition}

\theoremstyle{remark}

\numberwithin{equation}{section}

\usepackage[utf8]{inputenc}
\usepackage[russian,english]{babel}

\begin{document}

\title{On coefficients satisfying Chebyshev's approximation of $\pi(x)$}


\author{Connor Paul Wilson}
\address{530 Church Street
Ann Arbor, MI 48109}
\curraddr{Unit 7800 Box 339,
DPO, AP 96549}
\email{dpoae@umich.edu}
\thanks{}


\date{}

\dedicatory{I wish you were here Laura, descansa en el para\'iso ni\~na.}


\begin{abstract}
We note an interesting and under-expressed fact from Chebyshev's initial bounding for the prime counting function, $\pi(x) := \# \{p \leq x : p \text{ prime}\},$ based upon a selection of fixed coefficients $d\in D$ to show $\psi(x) \asymp x$, and thus the goal of choosing some $a(d)$ approximately the same as $\mu(d)$ such that:
    $$
    \sum_{d}\frac{a(d)}{d} = 0, \quad \wedge \quad -\sum_{d}\frac{a(d)\log d}{d} \approx 1.
    $$
\end{abstract}

\maketitle

\section{Introduction}

It is by Chebyshev~\cite{Chebyshev} that we take the following definitions, as
$$
\vartheta(x):=\sum_{p \leq x} \log p, \  \psi(x):=\sum_{1 \leq n \leq x} \Lambda(n)=\sum_{\substack{p^{m} \leq x \\ m \geq 1}} \log p,
$$
clearly for prime $p$, and $\Lambda(n)$ the von Mangoldt function.
Clearly, we have:
    $$
    \Lambda(n) = \sum_{d\mid n}\mu(d)\log\frac{n}{d}
    $$
by M\"obius inversion. Chebyshev took this simple fact and applied to it a general idea of relations such that for some:
    $$
    c(n) = \sum_{lm = n}a(l)b(m),
    $$
we can take a summatory function of $c(n)$ such that $C(x) = \sum_{n\leq x}c(n)$. We clearly thus also have $C(x) = \sum_{lm\leq x}a(l)b(m).$

Therefore,
    $$
    \psi(x) = \sum_{d\leq x}\mu(d)\sum_{n\leq\frac{x}{d}}\log n.
    $$

Note that we understand the summation term in the above equation as well, as if we take some $S(x) = \sum_{n\leq x}\log n,$ we know:
    $$
    S(x) = x\log x + O(\log x).
    $$
    
Thus it is here that we find Chebyshev's approach, essentially that he wishes to study $\psi(x)$ through approximations of the formulation using the M\"obius function. But by the same drawbacks that this approach gives us, we also take that producing bounds for any of these given functions of primes -- $\psi(x),$ $\pi(x),$ and $\vartheta(x)$ -- gives us a bound for any of them.

Now as did Chebyshev, let us aim to approximate $\psi(x)$ through $$ \psi(x) = \sum_{d\leq x}\mu(d)\sum_{n\leq\frac{x}{d}}\log n.$$ 
In order to do this, let us write a set of coefficients in place of the M\"obius function $\mu(d)$, $a(d)$. As did Chebyshev, let us fix our attention to a fixed, finite set of coefficients $D$ such that for $d\in D$ we have $a(d)\neq 0$.

Replacing $\mu(d)$ with $a(d),$ we obtain:
\begin{equation}
    \sum_{d\in D}a(d)S\left(\frac{x}{d}\right) = (x\log x-x)\sum_{d\in D}\frac{a(d)}{d} - x\sum_{d\in D}\frac{a(d)\log d}{d} + O(\log x).
\end{equation}

And thus our goal of approximating $\psi(x)$ using the method of Chebyshev becomes apparent: we wish to show $\psi(x) \asymp x$, and thus we have the goal of choosing some $a(d)$ approximately the same as $\mu(d),$ and thus such that:
    $$
    \sum_{d}\frac{a(d)}{d} = 0, \quad \wedge \quad -\sum_{d}\frac{a(d)\log d}{d} \approx 1.
    $$

\section{Chebyshev's initial bounds for $\pi(x)$}
Chebyshev stumbled upon the surprisingly accurate estimate initially in his research with $a(1) = 1$ and $a(2) = -2$, such that:
    $$
    \sum_{d}\frac{a(d)}{d} = 0, \quad \wedge \quad -\sum_{d}\frac{a(d)\log d}{d} = \log 2.
    $$
Let us derive an estimate from this selection of coefficients. Rewriting (1), we get:
\begin{equation}
    \begin{aligned}
    \sum_{d}a(d)S\left(\frac{x}{d}\right) &= \sum_{dn\leq x}a(d)\log n = \sum_{dn\leq x}a(d)\sum_{l\mid n}\Lambda(l) \\
    &= \sum_{dln\leq x}a(d)\Lambda(l) = \sum_{l\leq x}\left(\Lambda(l)\sum_{dn\leq \frac{x}{l}}a(d)\right).
    \end{aligned}
\end{equation}
It is clear that our estimate for $\psi(n)$ is closely related to the last summation $\sum_{dn\leq \frac{x}{l}}a(d) := E\left(\frac{x}{l}\right)$. Additionally, using $\sum_{d}\frac{a(d)}{d}$, we can extend the evident periodicity of $E\left(\frac{x}{l}\right)$ to:
    $$
    E(t) = \sum_{dn\leq t}a(d) = \sum_{d}a(d)\lfloor\frac{t}{d}\rfloor = -\sum_{d}a(d)\left\{\frac{t}{d}\right\},
    $$
giving us the fact that the expression must hold a period of $\operatorname{lcm}(d\in D)$.

Clearly, with our $D = \{1,2\}$ and our $a(d)$ conditions fixed as above, we obtain:
    $$
    E(t) = \begin{cases}
        0 & \text{ if } x= 0\leq t < 1\\ 
        1 & \text{ if } x= 1\leq t < 2,
    \end{cases}
    $$
which is the final piece of the puzzle needed to show an initial bound using Chebyshev's method.

We have:
    $$
    \psi(x)-\psi\left(\frac{x}{2}\right) = \sum_{\frac{x}{2}\leq l < x}\Lambda(l) \leq \sum_{l\leq x}\Lambda E\left(\frac{x}{l}\right) \leq \sum_{l\leq x}\Lambda(l) = \psi(x).
    $$
By (2), the central term can be expressed as:
    $$
    \sum_{l\leq x}\Lambda(l)E\left(\frac{x}{l}\right) = x(\log 2) + O(\log x),
    $$
which gives the bounds:
    $$
    \begin{aligned}
    \psi(x) - \psi\left(\frac{x}{2}\right) &\leq x(\log 2) + O(\log x) \\
    \psi(x) &\geq x(\log 2) + O(\log x),
    \end{aligned}    
    $$
where we obtain,
    $$
    \psi(x) = \sum_{0\leq r\leq (\log_{2}x) + 2}\left(\psi\left(\frac{x}{2^{r}}\right) - \psi\left(\frac{x}{2^{r+1}}\right)\right)\leq 2\log2x + O(\log^{2}x),
    $$
and thus,
    $$
    (\log2)x\leq \psi(x) + O(\log^{2}x)\leq (2\log2)x,
    $$
such that for sufficiently large $x$, we have
    $$
    (\log2)\frac{x}{\log x} < \pi(x) < (2\log2)\frac{x}{\log x}.
    $$

Chebyshev improved upon this method in his 1850 paper, using: 
    $$
    \begin{aligned}
    a(1) = a(30) &= 1 \\
    a(2) = a(3) = a(5) &= -1 \\
    a(d) &= 0, \forall \textup{ other } d,
    \end{aligned}
    $$
by the fact that,
    $$
    \frac{1}{1} - \frac{1}{2} - \frac{1}{3} - \frac{1}{5} + \frac{1}{30} = 0,
    $$
while
    $$
    \frac{\log2}{2} + \frac{\log3}{3} + \frac{\log5}{5} - \frac{\log30}{30} \approx 1.
    $$
to eventually obtain the famous:
    $$
    0.9212\frac{x}{\log x} < \pi(x) < 1.057\frac{x}{\log x}.
    $$
    
And it is here in which we find the titular coefficient selection mentioned in the abstract of this paper, that is to an incredible degree under-referenced in this field of study.
\section{Coefficients of approximation by Chebyshev}
It is clear that by the method of Chebyshev's approach, better selections of our $d\in D$ such that $a(d) \approx \mu(d)$ in formula (1.1). In our first approach to the problem of approximating $\pi(x)$ following Chebyshev's approach, we attain a relatively decent approximation using a set of only 2 coefficients, while his latter approximation is clearly much more accurate using 5.

We will state the following theorem about coefficients of Chebyshev's approach:

\begin{theorem}
Finding sets of numbers $a(d) := \{a_{1}, a_{2}, ... , a_{n}\}$ that provide an approximation of the prime counting function of the form:
$$
    \alpha\frac{x}{\log x} < \pi(x) < \beta\frac{x}{\log x},
$$
such that $\alpha, \beta \in \mathbb{R}$ using Chebyshev's method of approximation can be equated to finding our set $a(d)$ such that:
$$
    \sum_{k=1}^{n}\frac{1}{a_{k}} = 0
$$
and
$$
    \sum_{k=1}^{n}\frac{\log|a_{k}|}{k} \approx 1.
$$
\end{theorem}

It is also worth noting the value of $n$ in $a(d),$ such that when a solution to the conditions presented above is found, it is imperative that the `degree' of the solution under Chebyshev's approach is included (such that $a_{n}(d)$ is represented as the solution.) It is clearly the case that solutions of higher degree will necessarily bound solutions of lower degree in their approximation value, given the impact that the number of coefficients has on the function $E(t)$ such that:

$$
    \psi(x)-\psi\left(\frac{x}{2}\right) = \sum_{\frac{x}{2}\leq l < x}\Lambda(l) \leq \sum_{l\leq x}\Lambda E\left(\frac{x}{l}\right) \leq \sum_{l\leq x}\Lambda(l) = \psi(x).
$$

However, although research into finding better sets of these coefficients is limited-to-none, it's worth noting that this is of cursory interest, but not necessarily of genuine importance, as we will see by Diamond and Erd\"os~\cite{Erdos} that there is an approach to the problem such that an $a_{n}(d)$ exists which for all intents and purposes maintains the property:
$$
\lim_{n\rightarrow\infty}a_{n}(d),
$$

providing the arbitrarily better bound:
$$
    \limsup_{x \rightarrow \infty}\left| \frac{\pi(x)}{x/\log(x)} - 1 \right| < \varepsilon
$$

\section{Diamond and Erd\"os}

Let us formally state the approach of the Diamond-Erd\"osian method, as the production of arbitrarily better bounds than Chebyshev such that:

\begin{theorem}
    For $\varepsilon > 0, \exists T = T(\varepsilon) \in \mathbb{Z^{+}},$ such that the M\"obius function with values on the interval $\mu:[1, T)$ gives:
    $$
    \limsup_{x \rightarrow \infty}\left| \frac{\pi(x)}{x/\log(x)} - 1 \right| < \varepsilon
    $$
\end{theorem}

\begin{proof}
It is clear from our definition of $\psi$ that we can take:

$$
\limsup_{x \rightarrow \infty}\left|\frac{\psi(x)}{x}-1\right|
$$

as by the work of Chebyshev:

$$
\psi(x) = \sum_{p\leq x}\left\lfloor\frac{\log x}{\log p} \right \rfloor \log p\leq \sum_{p\leq x}\log x= \pi(x)\log x,
$$

such that,

$$
\pi(x) \leq \sum_{n\leq x} \frac{\Lambda(n)}{\log n} = \frac{\psi(x)}{\log x} + \int_{1}^{x}\frac{\psi(t)}{t\log^{2}t}\ dt \leq \frac{\psi(x)}{\log x} + \frac{Bx}{\log^{2}x},
$$

where our constant $B$ comes from Chebyshev's bound on $\psi(x) = O(x)$.

Following Diamond and Erd\"os, let us define over the strictly positive set of integers $T$:

$$
\mu_{T}(n) = \begin{cases}
\mu(n) & \text{ for } 1 \leq n <T, \\ 
-T\sum_{j<T}\frac{\mu(j)}{j} & \text{ for } n = T, \\ 
0 & \text{ for } n > T. 
\end{cases}
$$

In order to eventually show Theorem 4.1, we start with

\begin{lemma}
    $$
    \sum_{n\leq x} \Lambda \ast 1 \ast \mu_{T}(n) = \sum_{n\leq x} \Lambda \ast \mu_{T}(n)
    $$
approximates the statement
    $$
    \psi(x) = x
    $$
\end{lemma}

\begin{proof}
For some $\varepsilon > 0$, we take an unbounded sequence $T \in \mathbb{Z}^{+}$ such that:
    $$
    \left| \sum_{n\leq x} L \ast \mu_{T}(n) - x \right| < \varepsilon x, \text{ for } x(T)\leq x.
    $$
It is clear that for some $y \geq 1$ we have:
    $$
    \sum_{n\leq y}\log n = y\log y - y + O(\log ey),
    $$
and thus
    $$
    \begin{aligned}
    \sum_{n\leq x}L \ast \mu_{T}(n) &= \sum_{n\leq x} \sum_{ij = n}\log i\mu_{T}(j) = \sum_{ij\leq x}\log i\mu_{T}(j) \\
    &= \sum_{j\leq x}\left(\sum_{i\leq \frac{x}{j}}\log i\right)\mu_{T}(j) \\
    &= \sum_{j\leq x}\left(\frac{x}{j} (\log x -\log j - 1) + O\left(\frac{\log ex}{j}\right)\right)\mu_{T}(j) \\
    &= (x\log x - x)\sum_{j\leq x}\frac{\mu_{T}(j)}{j} - x\sum_{j\leq x}\frac{\log j}{j}\mu_{T}(j) \\
    &+ O((\log ex)\sum_{j\leq x}\left|\mu_{T}j\right|)
    \end{aligned}.
    $$
And clearly we have the value $O(T\log ex)$ for the $\varepsilon x$ in our lemma, so this just leaves us with the middle term.
Therefore,
    $$
    -\sum_{j\leq T}\frac{\log j}{j}\mu_{T}(j) = \sum_{j\leq T}\log\left(\frac{T}{j}\cdot\frac{\mu(j)}{j}\right),
    $$
such that,
    $$
    \int_{1}^{T}\log\frac{T}{u}\ d\sum_{i\leq u}\frac{\mu(i)}{i}.
    $$
We can take the average value of this integral by the fact that:
    $$
    s\int_{1}^{\infty}u^{-s-1}\ du = 1,
    $$
and therefore
    $$
    s\int_{1}^{\infty}u^{-s-1}\int_{1}^{u}\sum_{i\leq u}\frac{\mu(i)}{i}\frac{du}{u}\ du = \frac{1}{s\zeta(s+1)},
    $$
and thus we ensure our lemma by the bound of $\int_{1}^{n}\sum_{i\leq u}\frac{\mu(i)}{i}\frac{du}{u}$ for $n\in \mathbb{Z}.$
\end{proof}

In order to complete the proof of Theorem 4.1, let us use what Erd\"os called his ``hyperbolic method:"
    $$
    \sum_{n\leq x}\Lambda \ast m_{T}(n) = \sum_{ij\leq x}\Lambda(i)m_{T}(j),
    $$
for some $m_{T} := 1 \ast \mu_{T}$ such that $M_{T}(x) = \sum_{n\leq x}m_{T}(x)$.

Therefore,
    $$
    =\sum_{j\leq T-1}\psi\left(\frac{x}{j}\right)m_{T}(j) + \sum_{i\leq \frac{x}{T-1}}M_{T}\left(\frac{x}{i}\right)\Lambda(i) - M_{T}(T-1)\psi\left(\frac{x}{T-1}\right).
    $$

It is clear that we can simply use $\psi(x)$ to approximate the first term of this expression, and for the third $\psi(\frac{x}{T-1}) = O(\frac{x}{T})$ by the fact that $\mu_{T}(n) = \mu(n), \forall n < T$ gives $m_{T}(j) = 1 \ast \mu(j)$.
Thus all that is left is the middle term of the above expression. If we take some $K$ as a sufficiently large positive real number, then by Chebyshev's bound of $\psi(y) = O(y)$ and the prime number theorem we can take our estimate of $M_{T}(\frac{x}{i})$ such that:
    $$
    \left|\sum_{i\leq\frac{x}{T-1}}M_{T}\left(\frac{x}{i}\right)\Lambda(i)\right| \leq O\left(\frac{x}{K}\right) + \sum_{\frac{x}{TK}<i\leq\frac{x}{T-1}}\frac{T}{K}\Lambda(i),
    $$
which finally gives:
    $$
    = O\left(\frac{x}{K}\right),
    $$
allowing us to reduce the statement:
    $$
    \sum_{j\leq T-1}\psi\left(\frac{x}{j}\right)m_{T}(j) + \sum_{i\leq \frac{x}{T-1}}M_{T}\left(\frac{x}{i}\right)\Lambda(i) - M_{T}(T-1)\psi\left(\frac{x}{T-1}\right),
    $$
to
    $$
    \psi(x) + O\left(\frac{x}{T}\right) + O\left(\frac{x}{K}\right) = x + O(\varepsilon x),
    $$
for arbitrarily small $\varepsilon$.
\end{proof}

\vspace{0.5in}

\section{Acknowledgements}
The author would like to thank the Department of Mathematics at Brown University, as well as its associated NSF-funded partner \textit{Institute for Computational and Experimental Research in Mathematics} for its help in initially compiling related papers after the author's brief arXiv note on the topic~\cite{Wilson}; without this guidance it is very likely that an article in such form would have been composed, and the mathematics pertaining to the subject of this paper would remain largely bestrew. In particular, a Dr. David Lowry-Duda provided key insights into the contents of this paper, without which the author would be severely lacking.

The author would also like to pay his deepest respects and condolences to the family, friends, and loved ones of the late Laura L., who passed away tragically at the age of 16. The author also pays his condolences to the victims of adolescent suicide worldwide, particularly those from his hometown at W. T. Woodson High School, whose community has lost 7 young men and women to the same tragedy over the past 9 years. May they rest in peace, and let this paper serve to honor their memory.

\bibliographystyle{amsplain}

\end{document}